\documentclass[12pt,reqno]{amsart}
\usepackage[left=80pt,right=80pt]{geometry}
\usepackage[usenames]{color}
\usepackage{amsmath}
\usepackage{amssymb}
\usepackage{amsthm}
\usepackage{enumitem}
\usepackage{float}

\usepackage{graphicx}
\usepackage{tikz}
\usepackage{subcaption}
\usepackage{mathdots}

\usepackage{hyperref,url}
\hypersetup{
	colorlinks=true,
  linkcolor=black,          
  citecolor=black,         
  filecolor=black,      
  urlcolor=black           
}
\newtheorem{prop}{Proposition}
\newtheorem{lemma}[prop]{Lemma}
\newtheorem{theorem}[prop]{Theorem}

\newtheorem{corollary}[prop]{Corollary}
\theoremstyle{definition}

\newtheorem{example}[prop]{Example}

\newcommand{\W}{\mathcal{W}}

\newcommand{\Ss}{\mathcal{S}}
\newcommand{\seqnum}[1]{\href{https://oeis.org/#1}{\rm \underline{#1}}}
\allowdisplaybreaks

\begin{document}
\tikzset{mystyle/.style={matrix of nodes,
        nodes in empty cells,
        row 1/.style={nodes={draw=none}},
        row sep=-\pgflinewidth,
        column sep=-\pgflinewidth,
        nodes={draw,minimum width=1cm,minimum height=1cm,anchor=center}}}
\tikzset{mystyleb/.style={matrix of nodes,
        nodes in empty cells,
        row sep=-\pgflinewidth,
        column sep=-\pgflinewidth,
        nodes={draw,minimum width=1cm,minimum height=1cm,anchor=center}}}

\title{Counting $r\times s$ rectangles in nondecreasing and Smirnov words}

\author[SELA FRIED]{Sela Fried$^{*}$}
\thanks{$^{*}$ Department of Computer Science, Israel Academic College,
52275 Ramat Gan, Israel.
\\
\href{mailto:friedsela@gmail.com}{\tt friedsela@gmail.com}}

\begin{abstract}
The rectangle capacity, a word statistic that was recently introduced by the author and Mansour, counts, for two fixed positive integers $r$ and $s$, the number of occurrences of a rectangle of size $r\times s$ in the bargraph representation of a word. In this work we find the bivariate generating function for the distribution on nondecreasing words of the number of $r\times s$ rectangles and the generating function for their total number over all nondecreasing words. We also obtain the analog results for Smirnov  words, which are words that have no consecutive equal letters. This complements our recent results concerned with general words (i.e., not restricted) and Catalan words.
\bigskip

\smallskip

\noindent
\textbf{Math.~Subj.~Class.:} 05A05, 05A15.
\end{abstract}

\maketitle

\section{Introduction}
Let $n$ and $k$ be two positive integers and set $[k] = \{1,2,\ldots,k\}$. A word over $k$ of length $n$ is any element of the set $[k]^n$. Words have a visual representation in terms of bargraphs (see Figure \ref{fig;1} below), giving rise to many natural statistics on words, such as their water capacity, number of lit cells, and  perimeter (See \cite{MS2} for a comprehensive review of the subject).

In this work we continue the study of a new statistic, that was introduced by the author and Mansour in \cite{FrM}, called \emph{rectangle capacity}.  This statistic, for two fixed positive integers $r$ and $s$, counts the number of occurrences of a rectangle of size $r\times s$ in the bargraph representation of a word (see Figure \ref{fig;2} below). In \cite{FrM} we studied general words (i.e., not restricted) and Catalan words. The purpose of this work is to explore additional families of words, namely nondecreasing words and Smirnov words. By the former we mean words $w_1\cdots w_n\in[k]^n$ such that $w_{i+1}\geq w_i$, for each $1\leq i\leq n-1$, and, by the latter, we mean words $w_1\cdots w_n\in[k]^n$ such that $w_{i+1}\neq w_i$, for each $1\leq i\leq n-1$ (cf.\ \cite[Example 23 on p.\ 193]{F}).

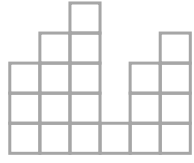
\begin{figure}[H]
\begin{center}
\centering
\scalebox{.40}{
\begin{tikzpicture}
\foreach \y in {1,2,3}
{\node[line width=1mm] at (0,\y) [rectangle,draw=black!35, minimum size=1cm]  {};}

\foreach \y in {1,2,3,4}
{\node[line width=1mm] at (1,\y) [rectangle,draw=black!35, minimum size=1cm]  {};}

\foreach \y in {1,2,3,4,5}
{\node[line width=1mm] at (2,\y) [rectangle,draw=black!35, minimum size=1cm]  {};}

\foreach \y in {1}
{\node[line width=1mm] at (3,\y) [rectangle,draw=black!35, minimum size=1cm]  {};}

\foreach \y in {1,2,3}
{\node[line width=1mm] at (4,\y) [rectangle,draw=black!35, minimum size=1cm]  {};}

\foreach \y in {1,2,3,4}
{\node[line width=1mm] at (5,\y) [rectangle,draw=black!35, minimum size=1cm]  {};}
\end{tikzpicture}}
\end{center}
\caption{The bargraph representation of the word $345134$}
\label{fig;1}
\end{figure}%

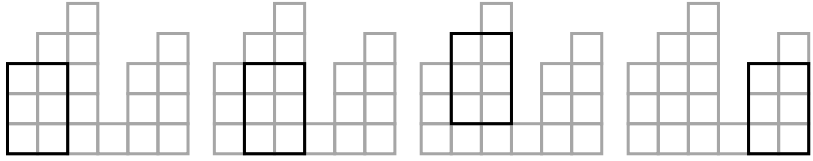
\begin{figure}[H]
\centering
\scalebox{.40}{
\begin{tikzpicture}
\foreach \y in {1,2,3}
{\node[line width=1mm] at (0,\y) [rectangle,draw=black!35, minimum size=1cm]  {};}

\foreach \y in {1,2,3,4}
{\node[line width=1mm] at (1,\y) [rectangle,draw=black!35, minimum size=1cm]  {};}

\foreach \y in {1,2,3,4,5}
{\node[line width=1mm] at (2,\y) [rectangle,draw=black!35, minimum size=1cm]  {};}

\foreach \y in {1}
{\node[line width=1mm] at (3,\y) [rectangle,draw=black!35, minimum size=1cm]  {};}

\foreach \y in {1,2,3}
{\node[line width=1mm] at (4,\y) [rectangle,draw=black!35, minimum size=1cm]  {};}

\foreach \y in {1,2,3,4}
{\node[line width=1mm] at (5,\y) [rectangle,draw=black!35, minimum size=1cm]  {};}
\node[line width=1mm] at (0.5,2) [rectangle,draw=black!35, minimum width=2cm, minimum height = 3cm, color=black]  {};
\end{tikzpicture}
\hspace{0.5cm}
\begin{tikzpicture}
\foreach \y in {1,2,3}
{\node[line width=1mm] at (0,\y) [rectangle,draw=black!35, minimum size=1cm]  {};}

\foreach \y in {1,2,3,4}
{\node[line width=1mm] at (1,\y) [rectangle,draw=black!35, minimum size=1cm]  {};}

\foreach \y in {1,2,3,4,5}
{\node[line width=1mm] at (2,\y) [rectangle,draw=black!35, minimum size=1cm]  {};}

\foreach \y in {1}
{\node[line width=1mm] at (3,\y) [rectangle,draw=black!35, minimum size=1cm]  {};}

\foreach \y in {1,2,3}
{\node[line width=1mm] at (4,\y) [rectangle,draw=black!35, minimum size=1cm]  {};}

\foreach \y in {1,2,3,4}
{\node[line width=1mm] at (5,\y) [rectangle,draw=black!35, minimum size=1cm]  {};}
\node[line width=1mm] at (1.5,2) [rectangle,draw=black!35, minimum width=2cm, minimum height = 3cm, color=black]  {};
\end{tikzpicture}
\hspace{0.5cm}
\begin{tikzpicture}
\foreach \y in {1,2,3}
{\node[line width=1mm] at (0,\y) [rectangle,draw=black!35, minimum size=1cm]  {};}

\foreach \y in {1,2,3,4}
{\node[line width=1mm] at (1,\y) [rectangle,draw=black!35, minimum size=1cm]  {};}

\foreach \y in {1,2,3,4,5}
{\node[line width=1mm] at (2,\y) [rectangle,draw=black!35, minimum size=1cm]  {};}

\foreach \y in {1}
{\node[line width=1mm] at (3,\y) [rectangle,draw=black!35, minimum size=1cm]  {};}

\foreach \y in {1,2,3}
{\node[line width=1mm] at (4,\y) [rectangle,draw=black!35, minimum size=1cm]  {};}

\foreach \y in {1,2,3,4}
{\node[line width=1mm] at (5,\y) [rectangle,draw=black!35, minimum size=1cm]  {};}
\node[line width=1mm] at (1.5,3) [rectangle,draw=black!35, minimum width=2cm, minimum height = 3cm, color=black]  {};
\end{tikzpicture}
\hspace{0.5cm}
\begin{tikzpicture}
\foreach \y in {1,2,3}
{\node[line width=1mm] at (0,\y) [rectangle,draw=black!35, minimum size=1cm]  {};}

\foreach \y in {1,2,3,4}
{\node[line width=1mm] at (1,\y) [rectangle,draw=black!35, minimum size=1cm]  {};}

\foreach \y in {1,2,3,4,5}
{\node[line width=1mm] at (2,\y) [rectangle,draw=black!35, minimum size=1cm]  {};}

\foreach \y in {1}
{\node[line width=1mm] at (3,\y) [rectangle,draw=black!35, minimum size=1cm]  {};}

\foreach \y in {1,2,3}
{\node[line width=1mm] at (4,\y) [rectangle,draw=black!35, minimum size=1cm]  {};}

\foreach \y in {1,2,3,4}
{\node[line width=1mm] at (5,\y) [rectangle,draw=black!35, minimum size=1cm]  {};}
\node[line width=1mm] at (4.5,2) [rectangle,draw=black!35, minimum width=2cm, minimum height = 3cm, color=black]  {};
\end{tikzpicture}}
\caption{There are four occurrences of a rectangle of size $3\times 2$ in the bargraph representation of the word $345134$.}\label{fig;2}
\end{figure}

Before we begin, let us fix three positive integers $r,s$, and $k$ and let $n$ be a nonnegative integer. If $m$ is a positive integer, we denote by $[m]$ the set $\{1,2,\ldots,m\}$. 

\section{Main results - Nondecreasing Words}

We denote by $\W_{n,k}^{\nearrow}$ the set of nondecreasing words over $k$ of length $n$, i.e.
\[\W_{n,k}^{\nearrow}=\left\{ w_{1}\cdots w_{n}\in[k]^{n}\;:\;w_{i}\leq w_{i+1}\text{ for every \ensuremath{i\in[n-1]}}\right\}.\] It is well known that \[\left|\W_{n,k}^{\nearrow}\right |=\binom{n+k-1}{k-1}.\] We distinguish between two cases, namely $r=1$ and $r\geq 2$.

\subsection{\texorpdfstring{$1\times s$}{} rectangles}

Denote by $a_{n,k}=a_{n,k}(t)$ the distribution on $\W_{n,k}^{\nearrow}$ of the number of $1\times s$ rectangles and let $A_k(x,t)$ denote the generating function of the numbers $a_{n,k}$. We shall need the following restrictions of $\W_{n,k}^{\nearrow}$: Let \[\W_{n,k}^{\nearrow,(0)}=\left\{ w_{1}\cdots w_{n}\in\W_{n,k}^{\nearrow}\;:\;w_{j}>1\text{ for every \ensuremath{j\in[n]}}\right\} \]
and, for $i\in [n]$, we define \[\W_{n,k}^{\nearrow,(i)}	=\left\{ w_{1}\cdots w_{n}\in\W_{n,k}^{\nearrow}\;:\;w_{i}=1\textnormal{ and } w_{j}>1\text{ for every \ensuremath{i<j\leq n}}\right\}.\] For $0\leq i\leq n$, let $a_{n,k}^{(i)}$ be the restriction of $a_{n,k}$ to $\W_{n,k}^{\nearrow,(i)}$. Notice that if $w_1\cdots w_n \in \W_{n,k}^{\nearrow,(i)}$, then $w_1=\cdots=w_i=1$. 

\begin{theorem}
We have
\begin{equation}\label{eq;p01}
A_k(x,t) =
\sum_{i=0}^{k-1}\frac{\alpha_{k-i}(t^{i}x,t)}{t^{i(s-1)}\prod_{j=1}^{i}(1-t^{j}x)},
\end{equation}
where \[\alpha_{m}(x,t)=\sum_{n=0}^{s-2}\binom{n+m-1}{m-1}x^{n}+\frac{x^{s-1}}{1-tx}\binom{s+m-3}{m-1}-\frac{1}{t^{s-1}(1-tx)}\sum_{n=0}^{s-2}\binom{n+m-2}{m-2}(tx)^{n}.\]
\end{theorem}

\begin{proof}
The set $\W_{n,1}^{\nearrow}$ consists solely of the word $1\cdots 1$, containing $n-s+1$ rectangles of size $1\times s$. Thus, $a_{n,1} = t^{\max\{0,n-s+1\}}$ and it is easily verified that the corresponding generating function is given by \[A_1(x,t) = \alpha_1(x,t) = \frac{1-x^{s-1}}{1-x}+\frac{x^{s-1}}{1-tx}.\] Assume now that $k\geq 2$. For $n\geq s$ we have
\begin{align*}
a_{n,k}&=a_{n,k}^{(0)}+\sum_{i=1}^{n}a_{n,k}^{(i)}\\&=t^{n-s+1}a_{n,k-1}+t^{n-s+1}\sum_{i=1}^{n}a_{n-i,k-1}.
\end{align*}
Multiplying both sides of this equation by $x^n$, summing over $n\geq s$ and adding $\sum_{n=0}^{s-1} a_{n,k}x^n$ to both sides, with some algebra we obtain \begin{align}
A_{k}(x,t)&=\sum_{n=0}^{s-1}\binom{n+k-1}{k-1}x^{n}-\binom{s+k-3}{k-2}x^{s-1}\nonumber\\&+\frac{1}{t^{s-1}(1-tx)}\left(A_{k-1}(tx,t)+\sum_{n=0}^{s-2}\binom{n+k-2}{k-2}((tx)^{s}-(tx)^{n})\right),\nonumber
\end{align} from which \eqref{eq;p01} follows by induction.
\end{proof}


\begin{corollary}\label{cor;fb1}
The generating function for the total number $f_k(n)$ of $1\times s$ rectangles over all words belonging to $\W_{n,k}^{\nearrow}$ is given by \[\frac{x^{s}\sum_{i=0}^{k-1}(-1)^{i}\binom{i+s-2}{i}\binom{k+s}{s+i+1}x^{i}}{(1-x)^{k+1}}.\] Thus, for $n\geq s$, \[f_k(n)=\sum_{i=0}^{\min\left\{ k-1,n-s\right\} }(-1)^{i}\binom{i+s-2}{i}\binom{k+s}{s+i+1}\binom{n-s-i+k}{k}.\]
\end{corollary}

\begin{proof}
Let $0\leq i\leq k-1$. We have 
\begin{align}
&\frac{\partial}{\partial t}\frac{\alpha_{k-i}(t^{i}x,t)}{t^{i(s-1)}\prod_{j=1}^{i}(1-t^{j}x)}\nonumber\\&=\frac{\left(\frac{\partial}{\partial t}\alpha_{k-i}(t^{i}x,t)\right)t^{i(s-1)}\prod_{j=1}^{i}(1-t^{j}x)-\alpha_{k-i}(t^{i}x,t)\left(\frac{\partial}{\partial t}t^{i(s-1)}\prod_{j=1}^{i}(1-t^{j}x)\right)}{\left(t^{i(s-1)}\prod_{j=1}^{i}(1-t^{j}x)\right)^{2}} \nonumber.  \end{align} Now, 
\begin{align}
\left[t^{i(s-1)}\prod_{j=1}^{i}(1-t^{j}x)\right]_{|t=1}&=(1-x)^i,\nonumber\\ 
\left[\frac{\partial}{\partial t}t^{i(s-1)}\prod_{j=1}^{i}(1-t^{j}x)\right]_{|t=1}&=i(1-x)^{i-1}\left((s-1)(1-x)-\frac{x(i+1)}{2}\right)\nonumber.
\end{align} It is not hard to see that
\begin{align}
\left[\alpha_{k-i}(t^{i}x,t)\right]_{|t=1} &= \begin{cases}
0&\textnormal{if }i < k-1\\ \frac{1}{1-x}&\textnormal{if }i = k-1,
\end{cases}\nonumber\\
\left[\frac{\partial}{\partial t}\alpha_{k-i}(t^{i}x,t)\right]_{|t=1}&=
\begin{cases}
\sum_{n=0}^{s-2}(s-1-n)\binom{n+k-1-i}{k-1-i}x^{n}&\textnormal{if }i < k-1\\ 
\frac{x(x^{s-1}+k-1)}{(1-x)^{2}}&\textnormal{if }i = k-1.
\end{cases}\nonumber
\end{align} It follows that 
\begin{align}
&\left[\frac{\partial}{\partial t}A_k(x,t)\right]_{|t=1}\nonumber\\& =
\sum_{i=0}^{k-1}\left[\frac{\partial}{\partial t}\frac{\alpha_{k-i}(t^{i}x,t)}{t^{i(s-1)}\prod_{j=1}^{i}(1-t^{j}x)}\right]_{|t=1}\nonumber\\
&=\sum_{i=0}^{k-2}\frac{\sum_{n=0}^{s-2}(s-1-n)\binom{n+k-1-i}{k-1-i}x^{n}}{(1-x)^{i}}+\frac{x\left(x^{s-1}+\frac{(k-1)(2s+k)}{2}\right)-(k-1)(s-1)}{(1-x)^{k+1}},\nonumber
\end{align}
from which, with some algebra, the assertion follows.
\end{proof}

\begin{example}
For $k=2$ we have 
\[
f_2(n)=\frac{3n^{2}+(7-4s)n+(s-1)(s-4)}{2}\] and Table \ref{tablea1} lists the corresponding sequences that we found in the On-Line Encyclopedia of Integer Sequences (OEIS) \cite{SL}. For $k=3$ we have
\[f_3(n)=\frac{2n^{3}-3(s-3)n^{2}+(s^{2}-10s+13)n+2(s-1)(s-3)}{2}\] and it seems that no corresponding sequence is registered in the OEIS.
\end{example}

\begin{table}[H]
\begin{center}
{\renewcommand{\arraystretch}{1.1}
\begin{tabular}{||c| c||c|c ||}
 \hline
 $s$& $f_2(n)$ & OEIS \\ [0.5ex]
 \hline\hline
 1  &$(3n^2-3n)/2$ &\seqnum{A045943} \\ 
  2  &$(3n^2 - n - 2)/2$ &\seqnum{A115067} \\ 
   3  &$(3n^{2}-5n-2)/2$ &\seqnum{A140090} \\ 
   4  &$(3n^{2}-9n)/2$ &\seqnum{A140091} \\ 
   5  &$(3n^{2}-13n+4)/2$ &\seqnum{A059845} \\ 
   6  &$(3n^{2}-17n+10)/2$ &\seqnum{A140672} \\ 
   7  &$(3n^{2}-21n+18)/2$ &\seqnum{A140673} \\ 
   8  &$(3n^2-25n+28)/2$ &\seqnum{A140674} \\ 
   9  &$(3n^2-29n+40)/2$ &\seqnum{A140675} \\ 
   10  &$(3n^2-33n+54)/2$ &\seqnum{A151542} \\ 
   11  &$(3n^2-37n+70)/2$ &\seqnum{A370238} \\ 
  \hline
\end{tabular}
\caption{The total number $f_2(n)$ of $1\times s$ rectangles over all words of length $n$, for several values of $s$.}\label{tablea1}}
\end{center}
\end{table}

\subsection{\texorpdfstring{$r\times s$}{} rectangles, where \texorpdfstring{$r\geq 2$}{}}

Denote by $b_{n,k}=b_{n,k}(t)$ the distribution on $\W_{n,k}^{\nearrow}$ of the number of $r\times s$ rectangles and let $B_k(x,t)$ denote the generating function of the numbers $b_{n,k}$. We shall need the following restriction of $\W_{n,k}^{\nearrow}$: For $m\in[k]$, let \[\W_{n,k}^{\nearrow,(\geq m)}=\left\{ w_{1}\cdots w_{n}\in\W_{n,k}^{\nearrow}\;:\;w_{j}\geq m\text{ for each \ensuremath{j\in[n]}}\right\},\] and let $b_{n,k}^{(\geq m)}$ be the restriction of $b_{n,k}$ to $\W_{n,k}^{\nearrow,(\geq m)}$. Let $B_k^{(\geq m)}(x,t)$ be the generating function of the numbers $b_{n,k}^{(\geq m)}$.

\begin{lemma}
Assume that $k\geq r-1$. Then
\begin{equation}\label{8bz}
B_k^{(\geq r-1)}(x,t) =
\sum_{i=0}^{k-r+1}\frac{\alpha_{k-i}^{(\geq r-1)}(t^ix,t)}{t^{i(s-1)}\prod_{j=0}^{i-1}(1-t^jx)},
\end{equation} where 
\begin{align}
& \alpha_m^{(\geq r-1)}(x,t)\nonumber\\&=\sum_{n=0}^{s-1}\binom{n+m-r+1}{m-r+1}x^{n}+\frac{x^{s}}{1-x}\binom{s+m-r}{s-1}-\frac{1}{t^{s-1}(1-x)}\sum_{n=0}^{s-1}\binom{n+m-r}{m-r}(tx)^{n}.\nonumber
\end{align}
\end{lemma}

\begin{proof}
The set $\W_{n,r-1}^{\nearrow,(\geq r-1)}$ consists solely of the word $(r-1)\cdots (r-1)$, containing no rectangles of size $r\times s$. Thus, $b_{n,r-1}^{(\geq r-1)} = 1$ and the  corresponding generating function is given by $B_{r-1}^{(\geq r-1)}(x,t)=1/(1-x) $. Assume now that $k\geq r$. For $n\geq s$ we have
\begin{align}
b_{n,k}^{(\geq r-1)}&=b_{n,k}^{(\geq r)}+\sum_{i=1}^{n}b_{i,r-1}^{(\geq r-1)}b_{n-i,k}^{(\geq r)}\nonumber\\
&=t^{n-s+1}b_{n,k-1}^{(\geq r-1)}+\sum_{i=1}^{n-s}t^{n-i-s+1 }b_{n-i,k-1}^{(\geq r-1)}+\sum_{i=n-s+1}^{n}b_{n-i,k-1}^{(\geq r-1)}.  \nonumber
\end{align}
Multiplying both sides of this equation by $x^n$, summing over $n\geq s$ and adding $\sum_{n=0}^{s-1}b_{n,k}^{(\geq r-1)}x^{n}$ to both sides, with some algebra we obtain the equation
\begin{align}
B_{k}^{(\geq r-1)}(x,t)&=\frac{1}{t^{s-1}(1-x)}B_{k-1}^{(\geq r-1)}(tx,t)+\sum_{n=0}^{s-1}\binom{n+k-r+1}{k-r+1}x^{n}\nonumber\\&-\frac{1}{t^{s-1}(1-x)}\sum_{n=0}^{s-1}\binom{n+k-r}{k-r}(tx)^{n}+\frac{x^{s}}{1-x}\sum_{n=0}^{s-1}\binom{n+k-r}{k-r},\nonumber
\end{align}
from which the statement immediately follows by induction.
\end{proof}

\begin{theorem}\label{thm;a0}
Assume that $k\geq r$. Then
\[
B_k(x,t) =
\frac{1}{(1-x)^{r-1}}\sum_{i=1}^{k-r+1}\frac{\alpha^{(\geq r-1)}_{k-i}(t^ix,t)}{t^{i(s-1)}\prod_{j=1}^{i-1}(1-t^jx)}+\beta_{k}(x,t),
\]
where 
\begin{align}
\beta_k(x,t)&=\sum_{i=0}^{s-1}\binom{i+k-1}{k-1}x^{i}+\frac{1}{(1-x)^{r-1}}\sum_{i=0}^{s-1}\binom{i+k-r}{k-r}x^{i}\nonumber\\&-\sum_{i=0}^{s-1}\binom{i+k-r}{k-r}x^{i}\sum_{n=0}^{s-1-i}\binom{n+r-2}{r-2}x^{n}-\frac{1}{t^{s-1}(1-x)^{r-1}}\sum_{i=0}^{s-1}\binom{i+k-r}{k-r}(tx)^{i}.\nonumber
\end{align}

\begin{proof}
For $n\geq s$ we have 
\begin{align*}
b_{n,k}&=b_{n,k}^{(\geq r)}+\sum_{i=1}^{n}b_{i,r-1}b_{n-i,k}^{(\geq r)}\\&=t^{n-s+1}b_{n,k-1}^{(\geq r-1)}+\sum_{i=1}^{n-s}t^{n-i-s+1}b_{i,r-1}b_{n-i,k-1}^{(\geq r-1)}+\sum_{i=n-s+1}^{n}b_{i,r-1}b_{n-i,k-1}^{(\geq r-1)}.
\end{align*}
Multiplying both sides of this equation by $x^n$, summing over $n\geq s$ and adding $\sum_{n=0}^{s-1} b_{n,k}x^n$ to both sides, with some algebra we obtain the equation

\begin{align}
&B_{k}(x,t)\nonumber\\&=\frac{1}{t^{s-1}(1-x)^{r-1}}B_{k-1}^{(\geq r-1)}(tx,t)+\sum_{i=0}^{s-1}\binom{i+k-1}{k-1}x^{i}+\frac{1}{(1-x)^{r-1}}\sum_{i=0}^{s-1}\binom{i+k-r}{k-r}x^{i}\nonumber\\&-\sum_{i=0}^{s-1}\binom{i+k-r}{k-r}x^{i}\sum_{n=0}^{s-1-i}\binom{n+r-2}{r-2}x^{n}-\frac{1}{t^{s-1}(1-x)^{r-1}}\sum_{i=0}^{s-1}\binom{i+k-r}{k-r}(tx)^{i},\nonumber
\end{align} from which, together with \eqref{8bz}, the statement immediately follows.
\end{proof}
\end{theorem}

\begin{corollary}
The generating function for the total number $g_{k}(n)$ of $r\times s$ rectangles over all words belonging to $\W_{n,k}^{\nearrow}$ is given by  \[\frac{x^{s}\sum_{i=0}^{k-r+1}(-1)^{i}\binom{i+s-2}{i}\binom{k-r+1+s}{s+i+1}x^{i}}{(1-x)^{k+1}}.\] Thus, for $n\geq s$, \[g_k(n)=\sum_{i=0}^{\min\left\{ k-r+1,n-s\right\} }(-1)^{i}\binom{i+s-2}{i}\binom{k-r+1+s}{s+i+1}\binom{n-s-i+k}{k}.\]
\end{corollary}

\begin{proof}
Let $i\in[k-r+1]$. We have 
\begin{align}
&\frac{\partial}{\partial t}\frac{\alpha^{(\geq r-1)}_{k-i}(t^{i}x,t)}{t^{i(s-1)}\prod_{j=1}^{i-1}(1-t^{j}x)}\nonumber\\&=\frac{\left(\frac{\partial}{\partial t}\alpha^{(\geq r-1)}_{k-i}(t^{i}x,t)\right)t^{i(s-1)}\prod_{j=1}^{i-1}(1-t^{j}x)-\alpha^{(\geq r-1)}_{k-i}(t^{i}x,t)\left(\frac{\partial}{\partial t}t^{i(s-1)}\prod_{j=1}^{i-1}(1-t^{j}x)\right)}{\left(t^{i(s-1)}\prod_{j=1}^{i-1}(1-t^{j}x)\right)^{2}} \nonumber.  \end{align} Now, 
\begin{align}
\left[t^{i(s-1)}\prod_{j=1}^{i-1}(1-t^{j}x)\right]_{|t=1}&=(1-x)^{i-1},\nonumber\\ 
\left[\frac{\partial}{\partial t}t^{i(s-1)}\prod_{j=1}^{i-1}(1-t^{j}x)\right]_{|t=1}&=i(1-x)^{i-2}\left((s-1)(1-x)-\frac{x(i-1)}{2}\right)\nonumber.
\end{align} It is not hard to see that
\begin{align}
\left[\alpha^{(\geq r-1)}_{k-i}(t^{i}x,t)\right]_{|t=1} &= \begin{cases}
0&\textnormal{if }i < k-r+1\\ \frac{1}{1-x}&\textnormal{if }i = k-r+1,
\end{cases}\nonumber\\
\left[\frac{\partial}{\partial t}\alpha^{(\geq r-1)}_{k-i}(t^{i}x,t)\right]_{|t=1}&=
\begin{cases}
\frac{1}{1-x}\sum_{n=0}^{s-2}(s-1-n)\binom{n+k-r-i}{k-r-i}x^{n}&\textnormal{if }i < k-r+1\\ 
\frac{x(k-r+1)}{(1-x)^2}&\textnormal{if }i = k-r+1,
\end{cases}\nonumber\\
\left[\frac{\partial}{\partial t}\beta_k(x,t)\right]_{|t=1}&=\frac{1}{(1-x)^{r-1}}\sum_{n=0}^{s-2}(s-1-n)\binom{n+k-r}{k-r}x^n\nonumber.
\end{align} It follows that 
\begin{align}
&\left[\frac{\partial}{\partial t}A_k(x,t)\right]_{|t=1}\nonumber\\& =\frac{1}{(1-x)^{r-1}}
\sum_{i=1}^{k-r+1}\left[\frac{\partial}{\partial t}\frac{\alpha^{(\geq r-1)}_{k-i}(t^{i}x,t)}{t^{i(s-1)}\prod_{j=1}^{i-1}(1-t^{j}x)}\right]_{|t=1}+\left[\frac{\partial}{\partial t}\beta_k(x,t)\right]_{|t=1}\nonumber\\
&=\frac{1}{(1-x)^{r-1}}\left(\sum_{i=1}^{k-r}\frac{\sum_{n=0}^{s-2}(s-1-n)\binom{n+k-r-i}{k-r-i}x^{n}}{(1-x)^{i}}+(k-r+1)\frac{x(k-r+2)-2(s-1)(1-x)}{2(1-x)^{k-r+2}}\right)\nonumber\\&+\frac{1}{(1-x)^{r-1}}\sum_{n=0}^{s-2}(s-1-n)\binom{n+k-r}{k-r}x^n\nonumber\\&=\frac{1}{(1-x)^{r-1}}\sum_{i=0}^{k-r}\frac{\sum_{n=0}^{s-2}(s-1-n)\binom{n+k-r-i}{k-r-i}x^{n}}{(1-x)^{i}}+(k-r+1)\frac{x(k-r+2)-2(s-1)(1-x)}{2(1-x)^{k+1}}\nonumber
\end{align}
from which, with some algebra, the assertion follows.
\end{proof}

\begin{example}
For $r=k=2$ and arbitrary $s$ we have 
\[
g_2(n)=\binom{n-s+2}{2}.\] For $r=s=2$ and arbitrary $k$ we have
\[g_k(n)=\frac{n-1}{n+1}\binom{k}{2}\binom{n-1+k}{k},\] which, for $k=3, 4$ corresponds to \seqnum{A077414} and \seqnum{A105938} in the OEIS, respectively. 
For $s=1$ and arbitrary $k$ and $r$ we have
\[g_k(n)=\binom{k-r+2}{2}\binom{n-1+k}{k}\] and Table \ref{tablea5} lists the corresponding sequences that we found in OEIS.

\begin{table}[H]
\begin{center}
{\renewcommand{\arraystretch}{1.1}
\begin{tabular}{||c| c||c|c ||}
 \hline
 $k$& $r$&  OEIS \\ [0.5ex]
 \hline\hline
 $3$  &$2$ &\seqnum{A027480} \\ 
 $4$  &$2$ &\seqnum{A033487} \\ 
 $5$  &$2$ &\seqnum{A266732} \\ 
 $6$  &$2$ &\seqnum{A240440} \\ 
 $7$  &$2$ &\seqnum{A266733} \\ 
 $4$  &$3$ &\seqnum{A050534} \\ 
 $5$  &$3$ &\seqnum{A253945} \\ 
 $6$  &$3$ &\seqnum{A271040} \\ 
  \hline
\end{tabular}
\caption{The total number $g_k(n)$ of $r\times 1$ rectangles over all words of length $n$, for several values of $k$ and $r$.}\label{tablea5}}
\end{center}
\end{table}

\end{example}

\section{Main results - Smirnov Words}

We denote by $\Ss_{n,k}$ the set of Smirnov words (e.g., \cite[Example 23 on p.\ 193]{F}), i.e.,
\[\Ss_{n,k}=\left\{ w_{1}\cdots w_{n}\in[k]^{n}\;:\;w_{i}\neq w_{i+1}\text{ for every \ensuremath{i\in[n-1]}}\right\}.\] Since there are no Smirnov words for $k=1$, we assume that $k\geq 2$. We distinguish between two cases, namely $r=1$ and $r\geq 2$.

\subsection{\texorpdfstring{$1\times s$}{} rectangles}

Denote by $c_{n,k}=c_{n,k}(t)$ the distribution on $\Ss_{n,k}$ of the number of $1\times s$ rectangles. Let $C_k(x,t)$ denote the generating function of the numbers $c_{n,k}$. We shall need the following restrictions of $\Ss_{n,k}$: Let \[\Ss_{n,k}^{(0)}=\left\{ w_{1}\cdots w_{n}\in\mathcal{S}_{n,k}\;:\;w_{j}>1\text{ for every \ensuremath{j\in[n]}}\right\} \] 
and, for $i\in [n]$, we define \[\Ss_{n,k}^{(i)}	=\left\{ w_{1}\cdots w_{n}\in\Ss_{n,k}\;:\;w_{i}=1\textnormal{ and } w_{j}>1\text{ for every \ensuremath{j\in[i-1]}}\right\}.\] For $0\leq i\leq n$, let $c_{n,k}^{(i)}$ be the restriction of $c_{n,k}$ to $\Ss_{n,k}^{(i)}$. Clearly, 
\begin{align}
\left|\Ss_{n,k}\right|&=
\begin{cases}
1&\textnormal{if } n = 0\\
k(k-1)^{n-1} & \textnormal{otherwise},
\end{cases}\nonumber\\
\left|\Ss_{n,k}^{(1)}\right|&=
\begin{cases}
1&\textnormal{if } n = 0\\
(k-1)^{n-1} & \textnormal{otherwise}.
\end{cases}\nonumber
\end{align}

\begin{theorem}
We have
\begin{equation}\label{eq;pp01}
C_2(x,t) =
\frac{2(t-1)x^{s}}{(1-tx)(1-x)}+\frac{1+x}{1-x}\end{equation}
and, for $k\geq 3$,
\begin{align}
C_{k}(x,t)&=\frac{(1+tx)\left(\gamma_{k}(x,t)+\delta_{k}(x,t)C_{k-1}(tx,t)\right)}{1+tx-txC_{k-1}(tx,t)},\label{eq;11g0}
\end{align} where
\begin{align}
&\gamma_k(x,t)=\frac{1+x-k(k-1)^{s-1}x^{s}}{1-(k-1)x}-\frac{1+tx-k(k-1)^{s-1}(tx)^{s}}{t^{s-1}(1-(k-1)tx)}, \nonumber\\&\delta_k(x,t)=\frac{1-(k-2)tx-(k-1)^{s-1}(tx)^{s}}{t^{s-1}(1-(k-1)tx)}-\frac{tx}{1+tx}\frac{(1+x)(1-((k-1)x)^{s-1})}{1-(k-1)x}.\nonumber  
\end{align}
\end{theorem}

\begin{proof}
For $n\geq 1$, the set $\Ss_{n,2}$ consists of two words, namely $1212\cdots$ and $2121\cdots$, containing exactly $n-s+1$ $1\times s$ rectangles, each. It is easily verified that the corresponding generating function is given by \eqref{eq;pp01}. Assume now that $k\geq 3$. For $n\geq s$ we have
\begin{align}
c_{n,k}&=c_{n,k}^{(0)}+\sum_{i=1}^{n}c_{n,k}^{(i)}\nonumber\\
&=c_{n,k-1}+\sum_{i=1}^{n-s+1}t^{i-1}c_{i-1,k-1}c_{n-i+1,k}^{(1)}+\sum_{i=n-s+2}^{n}t^{n-s+1}c_{i-1,k-1}c_{n-i+1,k}^{(1)},\label{eq;hg0}\\
c_{n,k}^{(1)}&=t\left(c_{n-1,k}-c_{n-1,k}^{(1)}\right).\label{eq;hg1}
\end{align}
Multiplying both sides of \eqref{eq;hg1} by $x^n$, summing over $n\geq s$ and adding $\sum_{n=0}^{s-1} c_{n,k}^{(1)}x^n$ to both sides, with some algebra we obtain 
\begin{equation}\label{eq;l01}
C_{k}^{(1)}(x,t)=\frac{1}{1+tx}\left(\sum_{n=0}^{s-1}c_{n,k}^{(1)}x^{n}+tx\left(C_{k}(x,t)-\sum_{n=0}^{s-2}c_{n,k}x^{n}\right)+tx\sum_{n=0}^{s-2}c_{n,k}^{(1)}x^{n}\right).    
\end{equation}
Multiplying both sides of \eqref{eq;hg0} by $x^n$, summing over $n\geq s$ and adding $\sum_{n=0}^{s-1} c_{n,k}x^n$ to both sides, with some algebra we obtain 
\begin{align}
C_{k}(x,t)&=\sum_{n=0}^{s-1}c_{n,k}x^{n}+\frac{1}{t^{s-1}}\left(C_{k-1}(tx,t)-\sum_{n=0}^{s-1}c_{n,k-1}(tx)^{n}\right)\nonumber\\&+C_{k-1}(tx,t)\left(C_{k}^{(1)}(x,t)-\sum_{n=0}^{s-1}c_{n,k}^{(1)}x^{n}\right)\nonumber\\ &+\frac{1}{t^{s-1}}\sum_{i=2}^{s}c_{s-i+1,k}^{(1)}(tx)^{s-i+1}\left(C_{k-1}(tx,t)-\sum_{n=0}^{i-2}c_{n,k-1}(tx)^{n}\right).  \label{eq;pp1}  
\end{align} Substituting \eqref{eq;l01} into \eqref{eq;pp1},  with some algebra we obtain \eqref{eq;11g0}.
\end{proof}

\begin{corollary}\label{cor;fb2}
The generating function for the total number $h_{k}(n)$ of $1\times s$ rectangles over all words belonging to $\Ss_{n,k}$ is given by 
\begin{equation}\label{yp1}
\frac{x^{s}\sum_{i=1}^{k-1}i^{s-1}(i+1)}{(1-(k-1)x)^{2}}.\end{equation} Thus, for $n\geq s$, \[h_{k}(n)=(k-1)^{n-s}(n-s+1)\sum_{i=1}^{k-1}i^{s-1}(i+1).\]
\end{corollary}

\begin{proof}
We have 
\[\left[\frac{\partial}{\partial t}C_2(x,t)\right]_{|t=1}=\frac{2x^{s}}{(1-tx)^{2}}.\] Thus, \eqref{yp1} holds for $k=2$. Assume now that $k\geq 3$ and that \eqref{yp1} holds for $k-1$. It is not hard to see that 
\begin{align}
\left[\gamma_k(x,t)\right]_{|t=1}&=0\nonumber,\\
\left[\frac{\partial}{\partial t}\gamma_{k}(x,t)\right]_{|t=1}&=\frac{x(s-1)+s}{1-(k-1)x}-\frac{1+x-k(k-1)^{s-1}x^{s}}{(1-(k-1)x)^{2}},\nonumber\\
\left[\delta_k(x,t)\right]_{|t=1}&=1\nonumber,\nonumber\\
\left[\frac{\partial}{\partial t}\delta_{k}(x,t)\right]_{|t=1}&=1-s-\frac{x}{1+x}\left[\frac{\partial}{\partial t}\gamma_{k}(x,t)\right]_{|t=1},\nonumber\\
\left[C_{k}(x,t)\right]_{|t=1}&=\frac{1+x}{1-(k-1)x},\nonumber\\
\left[\frac{\partial}{\partial x}C_{k}(x,t)\right]_{|t=1}&=\frac{k}{(1-(k-1)x)^{2}},\nonumber\\ 
\left[\frac{\partial}{\partial t}C_{k}(tx,t)\right]_{|t=1}&=x\left[\frac{\partial}{\partial x}C_{k}(x,t)\right]_{|t=1}+\left[\frac{\partial}{\partial t}C_{k}(x,t)\right]_{|t=1}.\nonumber
\end{align} Differentiating \eqref{eq;11g0} with respect to $t$ and substituting $t=1$, we obtain
\begin{align}
\left[\frac{\partial}{\partial t}C_{k}(x,t)\right]_{|t=1}&=\left[\frac{(1+tx)\left(\gamma_{k}(x,t)+\delta_{k}(x,t)C_{k-1}(tx,t)\right)}{1+tx-txC_{k-1}(tx,t)}\right]_{|t=1}\nonumber\\&=\frac{(s-1)(k-1)x^{2}+(2+(k-2)s)x-s+1}{(1-(k-1)x)^{2}}\nonumber\\&+\frac{(1-(k-2)x)^{2}}{(1-(k-1)x)^{2}}\left[\frac{\partial}{\partial t}C_{k-1}(x,t)\right]_{|t=1}+\left[\frac{\partial}{\partial t}\gamma_{k}(x,t)\right]_{|t=1}\nonumber\\&=\frac{k(k-1)^{s-1}x^{s}}{\left(1-(k-1)x\right)^{2}}+\frac{x^{s}\sum_{i=1}^{k-2}i^{s-1}(i+1)}{(1-(k-1)x)^{2}}\nonumber\\&=\frac{x^{s}\sum_{i=1}^{k-1}i^{s-1}(i+1)}{(1-(k-1)x)^{2}}.\nonumber\qedhere
\end{align}
\end{proof}

\begin{example}
For $k=2$ and arbitrary $s$ we have 
\[
h_2(n)=2(n-s+1).\] For $k=3$ and arbitrary $s$ we have
\[h_3(n)=2^{n-s}(n-s+1)(2+3\cdot2^{s-1}),\] which, for $s=2, 3$ corresponds to \seqnum{A241204} and \seqnum{A281200} in the OEIS, respectively.
\end{example}

\subsection{\texorpdfstring{$r\times s$}{} rectangles, where \texorpdfstring{$r\geq 2$}{}}

Denote by $d_{n,k}=d_{n,k}(t)$ the distribution on $\Ss_{n,k}$ of the number of $r\times s$ rectangles. Let $D_k(x,t)$ denote the generating function of the numbers $d_{n,k}$. We shall need the following restrictions of $\Ss_{n,k}$: For $m\in [k]$ let \begin{align}
\Ss^{(\geq m)}_{n,k}&=\left\{ w_{1}\cdots w_{n}\in\Ss_{n,k}\;:\;w_{i}\geq m\text{ for every \ensuremath{i\in[n]}}\right\},\nonumber\\
\bar{\Ss}^{(\geq m)}_{n,k}&=\left\{ w_{1}\cdots w_{n}\in\Ss^{(\geq m)}_{n,k}\;:\;w_1 \neq m\right\}.\nonumber\\
\bar{\Ss}_{n,k}&=\left\{ w_{1}\cdots w_{n}\in\Ss_{n,k}\;:\;w_1 \neq r-1\right\}.\nonumber
\end{align} Denote by $d^{(\geq m)}_{n,k}$ (resp.\ $\bar{d}^{(\geq m)}_{n,k}$, $\bar{d}_{n,k}$) the restriction of $d_{n,k}$ to $\Ss^{(\geq m)}_{n,k}$ (resp.\ to $\bar{\Ss}^{(\geq m)}_{n,k}$, $\bar{\Ss}_{n,k}$). Let $D^{(\geq m)}_k(x,t)$ (resp.\ $\bar{D}^{(\geq m)}_k(x,t)$, $\bar{D}_k(x,t)$) denote the generating function of the numbers $d^{(\geq m)}_{n,k}$ (resp.\ $\bar{d}^{(\geq m)}_{n,k}$, $\bar{d}_{n,k}$).
Clearly, 
\begin{align}
\left|\Ss^{(\geq m)}_{n,k}\right|&=
\begin{cases}
1&\textnormal{if } n = 0\\
(k-m+1)(k-m)^{n-1} & \textnormal{otherwise},
\end{cases}\nonumber\\
\left|\bar{\Ss}^{(\geq m)}_{n,k}\right|&=
(k-m)^n.\nonumber
\end{align}

\begin{lemma}
We have 
\begin{equation}\label{we51}
D_r^{(\geq r-1)}(x,t)=\frac{1+x}{1-x}\end{equation} and, for
$k\geq r+1$, \[
D_k^{(\geq r-1)}(x,t) =
(1+x)\frac{\frac{1}{t^{s-1}}D_{k-1}^{(\geq r-1)}(tx,t)+\sigma_{k}(x,t)}{-\frac{x}{t^{s-1}}D_{k-1}^{(\geq r-1)}(tx,t)+\rho_{k}(x,t)},
\] where 
\begin{align}
\sigma_{k}(x,t)&=\frac{1+x-(k-r+1)(k-r)^{s-1}x^{s}}{1-(k-r)x}-\frac{1+tx-(k-r+1)(k-r)^{s-1}(tx)^{s}}{t^{s-1}(1-(k-r)tx)},\nonumber\\
\rho_{k}(x,t)&=1-(k-r+1)x^{2}\frac{1-((k-r)x)^{s-1}}{1-(k-r)x}+x\frac{1+tx-(k-r+1)(k-r)^{s-1}(tx)^{s}}{t^{s-1}(1-(k-r)tx)}.\nonumber
\end{align}
\end{lemma}

\begin{proof}
For $n\geq 1$, the set $\Ss^{(\geq r-1)}_{n,r}$ consists of two words, namely $r(r-1)r(r-1)\cdots$ and $(r-1)r(r-1)r(r-1)\cdots$, containing no $r\times s$ rectangles. Thus, \eqref{we51} holds true in this case. Assume now that $k\geq r+1$. For $n\geq s$ we have
\begin{align}
d^{(\geq r-1)}_{n,k}&=d^{(\geq r)}_{n,k}+\sum_{i=1}^{n}d_{i-1,k}^{(\geq r)}\bar{d}_{n-i,k}^{(\geq r-1)}\nonumber\\
&=t^{n-s+1}d_{n,k-1}^{(\geq r-1)}+\sum_{i=1}^{s}d_{i-1,k-1}^{(\geq r-1)}\bar{d}_{n-i,k}^{(\geq r-1)}+\sum_{i=s+1}^{n}t^{i-s}d_{i-1,k-1}^{(\geq r-1)}\bar{d}_{n-i,k}^{(\geq r-1)},\label{yy1}\\
\bar{d}_{n,k}^{(\geq r-1)}&=d_{n,k}^{(\geq r)}+\sum_{i=2}^{n}d_{i-1,k}^{(\geq r)}\bar{d}_{n-i,k}^{(\geq r-1)}\nonumber\\&=t^{n-s+1}d_{n,k-1}^{(\geq r-1)}+\sum_{i=2}^{s}d_{i-1,k-1}^{(\geq r-1)}\bar{d}_{n-i,k}^{(\geq r-1)}+\sum_{i=s+1}^{n}t^{i-s}d_{i-1,k-1}^{(\geq r-1)}\bar{d}_{n-i,k}^{(\geq r-1)}.
\label{yy2}
\end{align}
Multiplying both sides of \eqref{yy1} by $x^n$, summing over $n\geq s$ and adding $\sum_{n=0}^{s-1}d_{n,k}^{(\geq r-1)}x^{n}$ to both sides, with some algebra we obtain the equation
\begin{align}
&D_{k}^{(\geq r-1)}(x,t)\nonumber\\&=\sum_{n=0}^{s-1}d_{n,k}^{(\geq r-1)}x^{n}+\frac{1}{t^{s-1}}\left(D_{k-1}^{(\geq r-1)}(tx,t)-\sum_{n=0}^{s-1}d_{n,k-1}^{(\geq r-1)}(tx)^{n}\right)-\sum_{i=1}^{s}d_{i-1,k-1}^{(\geq r-1)}x^{i}\sum_{n=0}^{s-i-1}\bar{d}_{n,k}^{(\geq r-1)}x^{n}\nonumber\\
&+\bar{D}_{k}^{(\geq r-1)}(x,t)\left(\sum_{i=1}^{s}d_{i-1,k-1}^{(\geq r-1)}x^{i}+\frac{x}{t^{s-1}}\left(D_{k-1}^{(\geq r-1)}(tx,t)-\sum_{n=0}^{s-1}d_{n,k-1}^{(\geq r-1)}(tx)^{n}\right)\right).\label{pr8}
\end{align}
Multiplying both sides of \eqref{yy2} by $x^n$, summing over $n\geq s$ and adding $\sum_{n=0}^{s-1}\bar{d}_{n,k}^{(\geq r-1)}x^{n}$ to both sides, with some algebra we obtain the equation
\begin{align}
&\bar{D}_{k}^{(\geq r-1)}(x,t)\left(1-\frac{x}{t^{s-1}}D_{k-1}^{(\geq r-1)}(tx,t)-\sum_{i=2}^{s}d_{i-1,k-1}^{(\geq r-1)}x^{i}+\frac{x}{t^{s-1}}\sum_{n=0}^{s-1}d_{n,k-1}^{(\geq r-1)}(tx)^{n}\right)\nonumber\\&=
\frac{1}{t^{s-1}}D_{k-1}^{(\geq r-1)}(tx,t)+\sum_{n=0}^{s-1}\bar{d}_{n,k}^{(\geq r-1)}x^{n}-\frac{1}{t^{s-1}}\sum_{n=0}^{s-1}d_{n,k-1}^{(\geq r-1)}(tx)^{n}\nonumber\\&-\sum_{i=2}^{s}d_{i-1,k-1}^{(\geq r-1)}x^{i}\sum_{n=0}^{s-i-1}\bar{d}_{n,k}^{(\geq r-1)}x^{n}.  \nonumber  
\end{align} Substituting this into \eqref{pr8}, with some algebra, the assertion follows.
\end{proof}

\begin{theorem}
We have \[D_r(x,t) =
\frac{1+x}{1-x}\] and, for $k\geq r+1$,
\begin{equation}\label{iw9}
D_k(x,t)=
(1+x)\frac{\frac{1}{t^{s-1}}D_{k-1}^{(\geq r-1)}(tx,t)+\sigma_{k}(x,t)}{-\frac{(r-1)x}{t^{s-1}}D_{k-1}^{(\geq r-1)}(tx,t)-(r-2)(1+x)+(r-1)\rho_{k}(x,t)}.
\end{equation}
\end{theorem}

\begin{proof}
We have
\begin{align}
d_{n,k}&=d_{n,k}^{(\geq r)}+(r-1)\sum_{i=1}^{n}d_{i-1,k}^{(\geq r)}\bar{d}_{n-i,k}\nonumber\\&=t^{n-s+1}d_{n,k-1}^{(\geq r-1)}+(r-1)\sum_{i=1}^{s}d_{i-1,k-1}^{(\geq r-1)}\bar{d}_{n-i,k}+(r-1)\sum_{i=s+1}^{n}t^{i-s}d_{i-1,k-1}^{(\geq r-1)}\bar{d}_{n-i,k}\label{gowp1},\\
\bar{d}_{n,k}&=d_{n,k}^{(\geq r)}+(r-1)\sum_{i=2}^{n}d_{i-1,k}^{(\geq r)}\bar{d}_{n-i,k}+(r-2)\bar{d}_{n-1,k}\nonumber\\&=t^{n-s+1}d_{n,k-1}^{(\geq r-1)}+(r-1)\sum_{i=2}^{s}d_{i-1,k-1}^{(\geq r-1)}\bar{d}_{n-i,k}\nonumber\\&+(r-1)\sum_{i=s+1}^{n}t^{i-s}d_{i-1,k-1}^{(\geq r-1)}\bar{d}_{n-i,k}+(r-2)\bar{d}_{n-1,k}\label{gowp2}.
\end{align}
Multiplying both sides of \eqref{gowp1} by $x^n$, summing over $n\geq s$ and adding $\sum_{n=0}^{s-1} d_{n,k}x^n$ to both sides, with some algebra we obtain 
\begin{align}
&D_{k}(x,t)\nonumber\\&=\sum_{n=0}^{s-1}d_{n,k}x^{n}+\frac{1}{t^{s-1}}\left(D_{k-1}^{(\geq r-1)}(tx,t)-\sum_{n=0}^{s-1}d_{n,k-1}^{(\geq r-1)}(tx)^{n}\right)-(r-1)\sum_{i=1}^{s}d_{i-1,k-1}^{(\geq r-1)}x^{i}\sum_{n=0}^{s-i-1}\bar{d}_{n,k}x^{n}\nonumber\\&+\bar{D}_{k}(x,t)\left((r-1)\sum_{i=1}^{s}d_{i-1,k-1}^{(\geq r-1)}x^{i}+\frac{(r-1)x}{t^{s-1}}\left(D_{k-1}^{(\geq r-1)}(tx,t)-\sum_{n=0}^{s-1}d_{n,k-1}^{(\geq r-1)}(tx)^{n}\right)\right).  \label{hio3} 
\end{align}
Multiplying both sides of \eqref{gowp2} by $x^n$, summing over $n\geq s$ and adding $\sum_{n=0}^{s-1} \bar{d}_{n,k}x^n$ to both sides, with some algebra we obtain 
\begin{align}
&\bar{D}_{k}(x,t)\nonumber\\&=\sum_{n=0}^{s-1}\bar{d}_{n,k}x^{n}+\frac{1}{t^{s-1}}\left(D_{k-1}^{(\geq r-1)}(tx,t)-\sum_{n=0}^{s-1}d_{n,k-1}^{(\geq r-1)}(tx)^{n}\right)\nonumber\\&+(r-1)\sum_{i=2}^{s}d_{i-1,k-1}^{(\geq r-1)}x^{i}\left(\bar{D}_{k}(x,t)-\sum_{n=0}^{s-i-1}\bar{d}_{n,k}x^{n}\right)\nonumber\\&+\frac{(r-1)x}{t^{s-1}}\bar{D}_{k}(x,t)\left(D_{k-1}^{(\geq r-1)}(tx,t)-\sum_{n=0}^{s-1}d_{n,k-1}^{(\geq r-1)}(tx)^{n}\right)+(r-2)x\left(\bar{D}_{k}(x,t)-\sum_{n=0}^{s-2}\bar{d}_{n,k}x^{n}\right).  \nonumber
\end{align} Solving for $\bar{D}_{k}(x,t)$ and substituting it into \eqref{hio3}, with some algebra the assertion follows.
\end{proof}

\begin{corollary}\label{cor;fb0}
The generating function for the total number $i_{k}(n)$ of $r\times s$ rectangles over all words belonging to $\Ss_{n,k}$ is given by \[\frac{x^{s}\sum_{i=1}^{k-r}i^{s-1}(i+1)}{(1-(k-1)x)^{2}}.\] Thus, for $n\geq s$, \[i_{k}(n)=(k-1)^{n-s}(n-s+1)\sum_{i=1}^{k-r}i^{s-1}(i+1).\]
\end{corollary}

\begin{proof}
It is not hard to see that
\begin{align}
\left[D_k^{(\geq r-1)}(x,t)\right]_{|t=1} &= \frac{1+x}{1-(k-r+1)x},\nonumber\\
\left[\frac{\partial}{\partial x}D_k^{(\geq r-1)}(x,t)\right]_{|t=1}&=\frac{k-r+2}{\left(1-\left(k-r+1\right)x\right)^{2}},\nonumber\\
\left[\frac{\partial}{\partial t}D_k^{(\geq r-1)}(x,t)\right]_{|t=1}&=\frac{x^{s}\sum_{n=0}^{k-r}n^{s-1}(n+1)}{(1-(k-r+1)x)^{2}}.\nonumber
\end{align} We omit the rest of the details, that are very similar to those in the proof of Corollary~\ref{cor;fb2}.
\end{proof}

\end{document}